\documentclass[11pt]{article}

\usepackage[margin=1in]{geometry} 
\usepackage{amsmath,amsthm,amssymb,amsfonts}
\usepackage{graphicx, multicol, array}
\usepackage{cases}
\usepackage{enumerate}
\usepackage{enumitem}
\usepackage{hyperref}
\usepackage{xcolor}

\usepackage{mathrsfs}

\newcommand{\N}{\mathbb{N}}
\newcommand{\Z}{\mathbb{Z}}

\newcommand{\R}{\mathbb{R}}

\newcommand{\tP}{\mathbb{P}}

\DeclareMathOperator{\tdiv}{div}

\newtheorem{thm}{Theorem}[section]
\newtheorem{lem}{Lemma}[section]
\newtheorem{conj}{Conjecture}[section]
\newtheorem{exa}{Example}[section]
\newtheorem{cor}{Corollary}[section]
\newtheorem{dfn}{Definition}[section]

\newtheorem{exe}{Exercise}[section]

\title{\textbf{Primes In Fractional Sequences}}
\author{N. A. Carella}
\date{}

\usepackage{fancyhdr}
\begin{document}
\thispagestyle{empty}
\maketitle

\vskip .25 in 

\begin{abstract} The results for the fractional sequence $\left \{[x/n]+1:n \leq x\right  \}$, and the fractional sequence in arithmetic progression 
$\left \{q[x/n]+a:n \leq x\right  \}$, where 
$a<q$ are integers such that $\gcd(a,q)=1$, prove that these sequences of fractional numbers contain the 
set of primes, and the 
set primes in arithmetic progressions as $x \to \infty$ respectively. Furthermore, the corresponding error terms for these sequences 
are improved.  
Other results considered are the fractional sequences of integers such as the sequence $\left \{[x/n]^2+1:n \leq x\right  \}$ generated 
by the quadratic polynomial $n^2+1$, 
and the sequence $\left \{[x/n]^3+2:n \leq x\right  \}$ generated by the cubic polynomial $n^3+2$. It is shown 
that each of these sequences of fractional numbers 
contains infinitely many primes as $x \to \infty$.\let\thefootnote\relax\footnote{Date: \today 
}
\end{abstract}

\tableofcontents

\vskip .25 in

\section{Introduction} \label{s437}
Let $x \geq 1$ be a large number, and let $[x]=x-\left \{x\right  \}$ denotes the largest integer function. Sequences of fractional numbers are of the forms
\begin{equation} \label{eq437.10}
\left \{\left [ s_{\beta}(n) \right ]: n \geq 1\right  \} \subset \N,
\end{equation}
where $s_{\beta}: \R \longrightarrow \R$ is a real-valued function. Some well known sequences of primes in fractional sequences are 
\begin{enumerate}
\item The sequence Beatty primes
\begin{equation} \label{eq437.11}
\left \{p=\left [ \alpha n \right ]: n \geq 1\right  \} \subset \tP,
\end{equation}
where $\alpha \in \R$ is an irrational number. The corresponding counting function is 
\begin{equation} \label{eq437.12} \pi_1(x)=\#\left \{p=\left [ \alpha n \right ]: p \leq x\right  \}=\delta(\alpha) \frac{x}{\log x}+O\left(\frac{x}{\log^2 x} \right ) ,
\end{equation}
where $\delta(\alpha)=1/\alpha>0$ is the density, see \cite{BS09}, \cite{SJ08}, \cite{HG16}.
\item The sequence  Piatetski-Shapiro primes
\begin{equation} \label{eq437.15}
\left \{p=\left [ n^{\beta} \right ]: n \geq 1\right  \}\subset \tP,
\end{equation}
where $\beta \in [1, 12/11]$ is a real number. The corresponding counting function is 
\begin{equation} \label{eq437.16}\pi_{\beta}(x)=\#\left \{p=\left [ n^{\beta} \right ]: p \leq x\right  \}=\delta(\beta) \frac{x^{1/\beta}}{\log x}+O\left(\frac{x^{1/\beta}}{\log^2 x} \right ) ,
\end{equation}
where $\delta(\beta)>0$ is the density, see \cite{KA99}, \cite{BR95}, \cite{HR95}, \cite{RS01}, \cite{MM13}. 
\end{enumerate}
The exponent of the integer variable $n \in \N$ ranges from linear in (\ref{eq437.12}) to subquadratic in (\ref{eq437.15}). The analysis for noninteger exponents 
are based on advanced analytic methods such as exponential sums, sieve methods, and circle methods, see \cite{KA99} and similar references. In contrast, the analysis for 
certain fractional sequences with integer exponents $\beta \geq 1$ can be achieved by elementary methods. \\

The results for the fractional sequences $\left \{[x/n]+1:n \leq x\right  \}$, and $\left \{q[x/n]+a:n \leq x\right  \}$, where 
$a<q$ are integers such that $\gcd(a,q)=1$, prove that these sequences of fractional numbers contain the set of primes, and the 
set primes in arithmetic progressions as $x \to \infty$ respectively. More, significantly, the corresponding error terms for these prime number theorems are improved. 
 \begin{thm} {\normalfont  }  \label{thm417.95} Let $x \geq 1$ be a large number and let $\Lambda$ be the vonMangoldt function. 
Then 
\begin{equation}
\sum_{n\leq x} \Lambda \left( [x/n]+1 \right )
=c_0x+O\left ( x^{(2+\varepsilon)/3}\log^2 x \right ),
\end{equation}
where the density constant is
\begin{equation}
c_0=\sum_{n\geq 1}\frac{ \Lambda \left( n+1 \right ) }{n(n+1) }\geq 0.7553658.
\end{equation}
\end{thm}

\begin{thm} {\normalfont  }  \label{thm427.95} Let $x \geq 1$ be a large number and let $a<q$ are integers such that $\gcd(a,q)=1$. let $\Lambda$ be the vonMangoldt function. 
Then 
\begin{equation}
\sum_{n\leq x} \Lambda \left( [qx/n]a \right )
=c(a,q)x+O\left ( x^{(2+\varepsilon)/3}\log^2 x \right ),
\end{equation}
where the density constant is
\begin{equation}
c(a,q)=\sum_{n\geq 1}\frac{ \Lambda \left( qn+a \right ) }{n(n+1) }>0.
\end{equation}
\end{thm}

These results can be viewed as new proofs of the asymptotic parts of the prime number theory and Dirichlet theorem for primes in arithmetic progressions. 
And the underlining technique is independent of the theory of the zeta function. Both the standard prime number theorem, and the prime number theorem in arithmetic progressions have subexponential error term, for example,
\begin{equation}
\sum_{n\leq x} \Lambda \left( n \right )
=x+O\left ( xe^{-\sqrt{\log x}} \right ),
\end{equation}
see \cite{MV07}. In contrast, the asymptotic results in Theorem \ref{thm417.95} and Theorem \ref{thm427.95} have sublinear error terms. \\


The next results prove that the fractional sequence 
\begin{equation} \label{eq437.32}
\left \{[x/n]^2+1:n \geq 1 \right  \} \subset \N,
\end{equation}
and the fractional sequence 
\begin{equation} \label{eq437.33}
\left \{[x/n]^3+2:n \geq 1 \right  \} \subset \N,
\end{equation} 
contains infinitely many primes as $x \to \infty$. In fact, Theorems \ref{thm437.05} and \ref{thm437.08} imply that the corresponding counting functions have the lower bounds
\begin{equation} \label{eq437.22}
\pi_{2}(x)= \#\left \{p= [x/n]^2+1 :p \leq x \text{ and }\right  \} \gg \frac{x^{1/2}}{\log x},
\end{equation}
and 
\begin{equation} \label{eq437.26}
\pi_{3}(x)=\#\left \{p= [x/n]^3+2 :p \leq x\text{ and }\right  \} \gg \frac{x^{1/3} }{\log x},
\end{equation}
as $ x\to \infty$, respectively.    
\begin{thm} {\normalfont  }  \label{thm437.05} Let $x \geq 1$ be a large number and let $\Lambda$ be the vonMangoldt function. Then 
\begin{equation}
\sum_{n\leq x} \Lambda \left( [x/n]^2+1 \right )
=a_2x+O\left ( x^{(2+\varepsilon)/3}\log^2 x \right ),
\end{equation}
where $\varepsilon >0$ is a small number and the density constant is
\begin{equation}
a_2=\sum_{n\geq 1}\frac{ \Lambda \left( n^2+1 \right ) }{n(n+1) }\geq 0.900076.
\end{equation}
\end{thm}
The estimate of the constant $a_2$ includes the smallest prime $p=1^2+1$.
\begin{thm} {\normalfont  }  \label{thm437.08} Let $x \geq 1$ be a large number and let $\Lambda$ be the vonMangoldt function. Then 
\begin{equation}
\sum_{n\leq x} \Lambda \left( [x/n]^3+2 \right )
=a_3x+O\left ( x^{(2+\varepsilon)/3}\log^2 x \right ),
\end{equation}
where $\varepsilon >0$ is a small number and the density constant is
\begin{equation}
a_3=\sum_{n\geq 1}\frac{ \Lambda \left( n^3+2 \right ) }{n(n+1) }\geq1.002998.
\end{equation}
\end{thm}

More generally, the number of primes in the fractional sequence
\begin{equation} \label{eq437.17}
\left \{p=f\left([x/n]\right ): n \leq x \right  \} \subset \tP,
\end{equation}
generated by an irreducible polynomial $f(t) \in \Z[t]$ is summarized in the following result.

\begin{thm} {\normalfont  }  \label{thm437.75} Let $f(t)$ be an irreducible polynomial of degree $\deg f=d=O(1)$, and fixed divisor $\tdiv(f)=1$. If $p=f(n)$ or $|f(n)|$ is prime for at least one integer $n \geq 1$, then 
\begin{equation}
\sum_{n\leq x} \Lambda \left(| f\left([x/n]\right )| \right )
=a_fx+O\left ( x^{(2+\varepsilon)/3}\log^2 x \right ),
\end{equation}
where $\varepsilon >0$ is a small number and the density constant is
\begin{equation}
a_f=\sum_{n\geq 1}\frac{ \Lambda \left( |f(n)| \right ) }{n(n+1) }> \frac{\log |f(n_0)|}{n_0(n_0+1)}>0,
\end{equation}
where $|f(n_0)|$ is the smallest prime in the integers sequence $\{f(n): n \geq 1\}$.
\end{thm}
 
The form of the counting functions (\ref{eq437.12}), (\ref{eq437.16}), (\ref{eq437.22}), and (\ref{eq437.26}) suggest the existence of a continuous interpolation scheme such that the prime counting functions for the generalized sequence  of Piatetski-Shapiro primes
\begin{equation} \label{eq437.43}
\left \{p=a\left [ n^{\beta} \right ]+b: n \geq 1\right  \}\subset \tP,
\end{equation}
where $\beta \geq 1$ is a real number and $a,b\geq 1$ are fixed integers, satisfy the chain of inequalities
\begin{equation} \label{eq437.44}
\pi(x)    \geq      \pi_{\alpha}(x)         \geq \pi_{2}(x)    \geq     \pi_{\beta}(x)     \geq \pi_{3}(x)       \geq \pi_{\gamma}(x) \geq \cdots,
\end{equation}
where $\pi(x)=\#\{p \leq x\}$, and the parameters $1 \leq \alpha \leq 2 \leq \beta \leq 2 \leq \gamma \leq 3 \leq \cdots$, are real numbers.\\

The preliminary notation, definitions and foundation are recorded in Sections \ref{s933}, and \ref{s773}. The proofs of Theorems \ref{thm437.05} and \ref{thm437.08} are assembled in Sections \ref{s437} and \ref{s439} respectively. Section \ref{s909} demonstrates an application of Theorem \ref{thm437.05}.


\section{Algebraic Foundation}   \label{lem933}
Elementary algebraic concepts used in the proofs of the main results are considered in this section. \\

The \textit{discrete image} or \textit{discrete spectrum} of a polynomial $f(x) \in \Z[x]$ over the integers is the subset of integers  $f(\Z) = \left \{ f(n) : n \in \Z \right  \}$.

\begin{dfn} \label{dfn933.11}  {\normalfont The \textit{fixed divisor} $\tdiv(f) = \gcd(f(\Z))$ of a polynomial $f(x)$ is the greatest common divisor of its image. The fixed divisor $\tdiv(f) = 1$ if the congruence equation $f(x) \equiv  0 \bmod p$ has $w(p) < p$ solutions for all primes $p\leq \deg(f)$, see \cite[p.\ 395]{FI10}. } 
\end{dfn}
For polynomials of small degrees $\deg f $, there is a fast method for computing the fixed divisor by means of a truncated image, and the greatest common divisor. Specifically,
\begin{equation}
\tdiv(f) = \gcd\left ( \left \{ f(n): 0 \leq n \leq \deg(f)\right  \} \right ). 
\end{equation}

The fixed divisors can also be computed from the coefficients of the polynomials. This procedure requires a change of basis from the power basis to the factorial basis.
\begin{lem} \label{lem933.12} {\normalfont (\cite[Lemma 1]{AO83})} The fixed divisor $\tdiv(f)$ of a polynomial $f(x) \in \Z[x]$ is given by $\tdiv(f)=\gcd(b_0,b_1, \ldots b_k)$, where 
\begin{equation}
f(x)=\sum_{0 \leq k \leq n} b_k \binom{x}{k}  \qquad \text{ and } \qquad \binom{x}{k} =\frac{x(x-1)(x-2) \cdots (x-k+1)}{k!}.
\end{equation}
\end{lem}

An integers sequence $\left \{s(n) : n \geq 1 \right \}$ or $\left \{f(n) : n \geq 1 \right \}$ generated by a real-valued function $s(x) \in C[\R]$ or a polynomial $f(x) \in \Z[x]$, whose prime spectrum $\left \{s(p): p\in \tP \right  \}$ or $\left \{f(p): p\in \tP\right  \}$ is infinite, generally have a trivial fixed divisor $\tdiv(s) = 1$ or $\tdiv(f) = 1$. But a sequence or polynomial whose prime spectrum is finite can have a nontrivial fixed divisor $\tdiv(s)\ne  1$ or $\tdiv(f) \ne  1$.  

\begin{exa} {\normalfont 
Some examples of irreducible polynomials with trivial fixed divisors. 
\begin{itemize}
\item $f_0(x)=qx+ a$, where $a<q$ are integers such that $\gcd(q,a)=1$, has fixed divisor $\tdiv(f_0)= 1$.
\item $f_1(x)=x^2+ 1$ has fixed divisor $\tdiv(f_1)= 1$.
\item $f_2(x) = x^2+x+ 1$ has fixed divisor $\tdiv(f_2)= 1$.
\item $f_3(x) = x^3 + 2$ has fixed divisor $\tdiv(f_3)=1$.
\end{itemize}
}
\end{exa} 
Since the fixed divisors $\tdiv(f)= 1$, these irreducible polynomials can generate infinitely many primes. But, irreducible polynomials with nontrivial fixed divisors $\tdiv(f) > 1$ can generate at most one prime.
\begin{exa} {\normalfont 
Some examples of irreducible polynomials with nontrivial fixed divisors. 
\begin{itemize}
\item $f_4(x)=x(x + 1)+2$ has fixed divisor $\tdiv(f_4)= 2$.
\item $f_5(x) = x(x + 1)(x + 2)  + 3$ has fixed divisor $\tdiv(f_5)= 3$.
\end{itemize}
}
\end{exa} 

A reducible polynomial $f(x)$ of degree $d = \deg(f)$ can generates up to $d \geq 1$ primes. For example, the reducible polynomial $f(x) = (g(x)\pm  1)(x^2 + 1)$ can generates up to $\deg(g) \geq  1$ primes or a finite number, for some polynomial $  g(x) \in  \Z[x]$ of degree $\deg(g) \geq   1$, see Exercise \ref{exe933.90}.

\section{Analytic Foundation}  \label{s773}
Elementary analytic concepts used in the proofs of the main results are considered in this section. \\

The vonMangoldt function is defined by the weighted prime powers indicator function
\begin{equation} \label{eq773.52}
\Lambda(n)=
\begin{cases}
\log p & \text{if } n=p^m,\\ 
0 & \text{if } n\ne p^m.\\ 
\end{cases}
\end{equation}
The symbol $p^m\geq 2$, with $ m \in \N$, denotes a prime power.\\

\begin{thm} {\normalfont(\cite[Theorem 2.6]{BS18})}  \label{thm773.01} Let $f$ be a complex-valued arithmetic function and assume that there exists $0<\alpha<2$ such that 
\begin{equation}
\sum_{n\leq x}\left | f(n) \right |^2 \ll x^{\alpha}.
\end{equation}
Then
\begin{equation}
\sum_{n\leq x}f \left( [x/n] \right ) =x\sum_{n\geq 1}\frac{f(n)}{n(n+1)}+O\left ( x^{(\alpha+1)/3}\log^{2}x \right ).
\end{equation}
\end{thm}

This result provides a different and simpler method for proving the existence of primes in some fractional sequences of real numbers \ref{eq437.12} of integer degrees $\beta \geq 1$. And the technique is independent of the theory of the zeta function. Furthermore, it probably can be used to interpolate to the noninteger exponents sequences such as (\ref{eq437.15}).

\begin{lem} \label{lem773.01} Let $x \geq 1$ be a large number. Let $g(t) \in \Z[t]$ be a polynomial of degree $\deg g=d$=O(1), and let $\Lambda$ be the vonMangoldt function. Then
\begin{enumerate} [font=\normalfont, label=(\roman*)]
\item  \begin{equation} \label{eq773.40}
\sum_{n\leq x}\Lambda \left( |g\left ([x/n]\right )| \right ) \ll x^{\alpha},
\end{equation}
\item  \begin{equation} \label{eq773.43}
\sum_{n\leq x}\Lambda \left(| g\left ([x/n]\right ) |\right )^2 \ll x^{\alpha},
\end{equation}
\end{enumerate}
where $\alpha=1+\varepsilon$ for any small number $\varepsilon>0$.
\end{lem}
\begin{proof} (i) For any large number $ x\geq 1$, and a polynomial $g(t)$ of degree $\deg g =O(1)$, use a series of estimates to compute a larger estimate:
\begin{eqnarray}  \label{eq773.45}
\sum_{n\leq x}\Lambda \left( |g\left ([x/n]\right )| \right )&\ll& \sum_{n\leq x}\Lambda\left( n \right )\log n\nonumber\\
&\ll& (\log^2 x)\sum_{n\leq x}1\\ 
&\ll& x\log^2 x\nonumber\\
&\ll& x^{\alpha}\nonumber
\end{eqnarray}
where $\alpha=1+\varepsilon$ for any small number $\varepsilon>0$. (ii) The verification is similar.
\end{proof}

\begin{lem} \label{lem773.24} Let $x \geq 1$ be a large number. Let $g_1(t), g_2(t) \in \Z[t]$ be polynomials of degrees $\deg g_i=d_i$=O(1), and let $\Lambda$ be the vonMangoldt function. Then
\begin{enumerate} [font=\normalfont, label=(\roman*)]
\item  \begin{equation} \label{eq773.50}
\sum_{n\leq x}\Lambda \left( |g_1\left ([x/n]\right )| \right ) \Lambda \left( |g_2\left ([x/n]\right )| \right )\ll x^{\alpha},
\end{equation}
\item  \begin{equation} \label{eq773.53}
\sum_{n\leq x}\Lambda \left( |g_1\left ([x/n]\right )| \right )^2 \Lambda \left( |g_2\left ([x/n]\right )| \right )^2 \ll x^{\alpha},
\end{equation}
\end{enumerate}
where $\alpha=1+\varepsilon$ for any small number $\varepsilon>0$.
\end{lem}

\section{Primes In Fractional Sequences Of Degree 1}  \label{s417}
The sequence of integers $\left \{ n+1 :n \geq 1\right  \} \subset \N$ and the associated subsequence of primes has an extensive literature, well known as the prime number theorem. 
But, there is no literature on the fractional sequence of integers $\left \{ [x/n]+1 :n \geq 1\right  \} \subset \N$. 

\begin{proof} (Theorem \ref{thm417.95}) Let $f(n)=\Lambda \left( n+1 \right ).$ Then, the condition
\begin{equation}
\sum_{n\leq x}\left | f(n) \right |^2=  \sum_{n\leq x}\Lambda^2 \left( n+1 \right ) \leq x \log^2 x  \ll x^{\alpha}
\end{equation}
is satisfied with $\alpha=1+\varepsilon$ for any small number $\varepsilon>0$. Applying Theorem \ref{thm773.01} yield
\begin{equation}
\sum_{n\leq x} \Lambda \left( [x/n]+1 \right )
=x\sum_{n\geq 1}\frac{ \Lambda \left( n+1 \right ) }{n(n+1) }+O\left ( x^{(2+\varepsilon)/3}\log^2 x \right ).
\end{equation}
By Euclid theorem, there are infinitely many primes. Hence the density constant is given by the infinite series
\begin{eqnarray}
c_0&=&\sum_{n\geq 1}\frac{ \Lambda \left( n+1 \right ) }{n(n+1)} \nonumber\\
&\geq &\frac{ \Lambda \left(1+1 \right ) }{1 (1+1)} \\
&>&0 \nonumber,
\end{eqnarray}
where $n+1=2$ is the first prime in the sequence of primes $ n+1$ for $n \geq 1$.
\end{proof}
A small scale numerical experiment gives the estimate
\begin{equation}
c_0\geq \sum_{n\leq 100000}\frac{ \Lambda \left( n+1 \right ) }{n(n+1)} \geq 0.755365841685897442410689.
\end{equation}

\section{Primes In Fractional Arithmetic Progressions Of Degree 1}  \label{s427}
The arithmetic progression of integers $\left \{ qn+a :n \geq 1\right  \} \subset \N$, where $a<q$ are integers such that $\gcd(a,q)=1$, 
and the associated subsequence of primes in arithmetic progression has an extensive literature, well known as Dirichlet theorem. 
But, there is no literature on the fractional sequence of integers $\left \{ a[x/n]+a :n \geq 1\right  \} \subset \N$. 

\begin{proof} (Theorem \ref{thm427.95}) Let $f(n)=\Lambda \left( qn+a \right ).$ Then, the condition
\begin{equation}
\sum_{n\leq x}\left | f(n) \right |^2=  \sum_{n\leq x}\Lambda^2 \left( qn+a \right ) \leq x \log^2 x  \ll x^{\alpha}
\end{equation}
is satisfied with $\alpha=1+\varepsilon$ for any small number $\varepsilon>0$. Applying Theorem \ref{thm773.01} yield
\begin{equation}
\sum_{n\leq x} \Lambda \left(q [x/n]+a\right )
=x\sum_{n\geq 1}\frac{ \Lambda \left(qn+a \right ) }{n(n+1) }+O\left ( x^{(2+\varepsilon)/3}(\log^2 x) \right ).
\end{equation}
By Dirichlet theorem, there are infinitely many primes in arithmetic progressions. Hence the density constant is given by the infinite series
\begin{eqnarray}
c(a,q)&=&\sum_{n\geq 1}\frac{ \Lambda \left(q n+a \right ) }{n(n+1)} \nonumber\\
&\geq &\frac{ \Lambda \left(qn_0+a \right ) }{n_0 (n_0+1)} \\
&>&0 \nonumber,
\end{eqnarray}
where $qn_0+a \geq 2$ is the first prime in the arithmetic progression $q n+a$ for $n \geq 1$.
\end{proof}

\section{Primes In Fractional Sequences Of Degree 2}  \label{s437}
The sequence of integers $\left \{ n^2+1 :n \leq x\right  \} \subset \N$ and the associated subsequence of primes has an extensive literature.  Detailed discussions of the prime values of polynomials appear in \cite[p.\ 48]{HL23}, \cite[p. 405]{RP96}, \cite[p. 395]{FI10}, \cite[p. 342]{NW00}, \cite[p. 33]{PJ09}, et alii.  However, there is no literature on the fractional sequence of integers $\left \{ [x/n]^2+1 :n \leq x\right  \} \subset \N$. A few results for primes in this fractional sequence are proved here.

\begin{proof} (Theorem \ref{thm437.05}) Let $f(n)=\Lambda \left([x/ n]^2+1 \right )$. Then, the condition
\begin{equation}
\sum_{n\leq x}\left | f(n) \right |^2  \ll x^{\alpha}
\end{equation}
is satisfied with $\alpha=1+\varepsilon$ for any small number $\varepsilon>0$. This is Lemma \ref{lem773.01} applied to the irreducible polynomial $g(t)=t^2+1$ of divisor $\tdiv g=1$. Applying Theorem \ref{thm773.01} yield
\begin{equation} \label{eq437.10}
\sum_{n\leq x} \Lambda \left( [x/n]^2+1 \right )
=x\sum_{n\geq 1}\frac{ \Lambda \left( n^2+1 \right ) }{n(n+1) }+O\left ( x^{(2+\varepsilon)/3}\log^2 x \right ).
\end{equation}
The density constant has the lower bound
\begin{eqnarray} \label{eq437.12}
a_2&=&\sum_{n\geq 1}\frac{ \Lambda \left( n^2+1 \right ) }{n(n+1)} \nonumber\\
&\geq &\frac{ \Lambda \left( n_0^2+1 \right ) }{n_0 (n_0+1)} \\
&=&\frac{\log 2}{2} \nonumber,
\end{eqnarray}
where $n_0^2+1=2$ is the first prime in the sequence of primes $ n^2+1$ for $n \geq 1$. Now, on the contrary, suppose that there are finitely many primes $p=[x/n]^2+1$. Then, there exists a large constant $x_0\geq 1$ such that
\begin{eqnarray} \label{eq437.20}
x_0 \log x_0& \geq &\sum_{n\leq x} \Lambda \left( [x/n]^2+1 \right ) \\
&\geq& 0.900076x+O\left ( x^{(2+\varepsilon)/3}\log^2 x \right ) \nonumber.
\end{eqnarray}
But, since the right side is unbounded as $x \to \infty$, this is a contradiction. Equivalently, it contradicts Theorem \ref{thm773.01}. Therefore, there are infinitely many primes of the form $p=[x/n]^2+1$.
\end{proof}
A small scale numerical experiment gives the estimate
\begin{equation}
a_2\geq \sum_{n\leq 100}\frac{ \Lambda \left( n^2+1 \right ) }{n(n+1)} \geq 0.900076.
\end{equation}
This estimate includes the smallest prime $p=1^2+1$. \\

The fractional prime counting function is defined by
\begin{equation}
\pi_{2}(x)= \#\left \{p= [x/n]^2+1 :p \leq x\text{ and } n \leq x\right  \}.
\end{equation}

\begin{cor} {\normalfont  }  \label{cor437.08}  The fractional prime counting function has the asymptotic formula
\begin{equation}
\pi_{2}(x) \gg\frac{x^{1/2}}{\log x}+O\left ( x^{(2+\varepsilon)/6}\log x \right ).
\end{equation}
\end{cor}

\begin{proof} An application of Theorem \ref{thm437.05}, and partial summation yield
\begin{eqnarray} \label{eq437.30}
\pi_{2}(x)& =&\sum_{\substack{p\leq x\\
p=[x/n]^2+1}}1\\
&=&\sum_{n\leq x^{1/2}}        \frac{ \Lambda \left( [x/n]^2+1 \right )}{\log \left( [x/n]^2+1 \right ) }  \nonumber\\
&\gg&\frac{x^{1/2}}{\log x}+O\left ( x^{(2+\varepsilon)/6}\log x \right ) \nonumber.
\end{eqnarray}
This proves the claim.
\end{proof}

\section{The Euler-Landau Problem} \label{s909}
As early as 1760, Euler was developing the theory prime values of polynomials. In fact, Euler computed a large table of the primes $p=n^2+1$, confer \cite[p.\ 123]{EL00}. Likely, the prime values of polynomials was studied by other authors before Euler. Later, circa 1910, Landau posed an updated question of the same problem about the primes values of the polynomial $n^2+1$, see \cite{PJ09}. A fully developed conjecture, based on circle methods analysis, appeared some time later.

\begin{conj} {\normalfont (\cite{HL23})} Let $x \geq 1$ be a large number. Let $\Lambda$ be the vonMangoldt function, and let $\chi$ be the quadratic symbol. Then
\begin{equation} \label{eq909.10}
\sum_{n\leq x} \Lambda \left( n^2+1 \right )
=c_2x+O\left (\frac{ x}{\log x }\right ),
\end{equation}
where the density constant
\begin{equation} \label{eq909.12}
c_2=\prod_{p \geq 3}\left( 1-\frac{\chi(-1)}{p-1} \right )=1.37281346 \ldots .
\end{equation}
\end{conj}
The general circle methods heuristics for admissible quadratic polynomials was proposed in \cite[p.\ 46]{HL23}. More recent discussions are given in \cite[p. 406]{RP96}, \cite[p.\ 342]{NW00}, et cetera. Some partial results are proved in \cite{DI82}, \cite{GM00}, \cite{BZ07}, and the recent  literature. Here, an application of Theorem \ref{thm437.05} proves the correct asymptotic order.
 \begin{cor} {\normalfont  }  \label{cor909.34}  Let $x \geq 1$ be a large number, and let $\Lambda$ be the vonMangoldt function. Then
\begin{equation} \label{eq909.20}
\sum_{n\leq x} \Lambda \left( n^2+1 \right )
\gg x+O\left ( x^{(2+\varepsilon)/3}\log^2 x \right ),
\end{equation}
where $\varepsilon >0$ is a small number.
\end{cor}

\begin{proof} For any large number $ x\geq 1$, the subsets relation
\begin{equation}\label{eq773.77}
\{ [x/n]: n \leq x\} \subset \{  n \leq x\}  
\end{equation}
holds. Hence, by Theorem \ref{thm773.01}, it follows that
\begin{eqnarray} \label{eq909.23}
\sum_{n\leq x} \Lambda \left( n^2+1 \right )
&\gg &\sum_{n\leq x} \Lambda \left( [x/n]^2+1 \right )\\
&\gg &a_2x+O\left ( x^{(2+\varepsilon)/3}\log^2 x \right ) \nonumber,
\end{eqnarray}
where $a_2>\log 2/2$. This proves the correct asymptotic order as claimed. 
\end{proof}

The conjectured density $c_2$ is well above the density
\begin{eqnarray} \label{eq909.25}
a_2&=&\sum_{n\geq 1}\frac{\Lambda\left( n^2+1 \right )}{n(n+1)}\\
&=& \sum_{n\leq 100}\frac{\Lambda\left( n^2+1 \right )}{n(n+1)}+\sum_{n> 100}\frac{\Lambda\left( n^2+1 \right )}{n(n+1)}\nonumber\\
&\leq &1.000000+\int_{100}^{\infty}\frac{\log t}{t(t+1)}dt\nonumber\\ 
&\leq &1.20000\nonumber\\
&< &c_2\nonumber.
\end{eqnarray}
Moreover, both of these finite sums have the same asymptotic order. Accorddingly, it is an interesting problem to determine the linking constant $c_L>0$ for which

\begin{equation} \label{eq909.28}
\sum_{n\leq x} \Lambda \left( n^2+1 \right )=c_L
\sum_{n\leq x} \Lambda \left( [x/n]^2+1 \right ).
\end{equation}

\section{ Primes In Fractional Sequences Of Degree 3}  \label{s439}
The sequence of integers $\left \{ n^3+2 :n \geq 1\right  \} \subset \N$ and the associated subsequence of primes has an extensive literature, see  \cite[p.\ 50]{HL23}, \cite[p.\ 407]{RP96}, \cite[p.\ 349]{NW00}, \cite[Chapter 12]{HG07}, et alii. However, there is no literature on the fractional sequence of integers $\left \{ [x/n]^3+2 :n \leq x\right  \} \subset \N$. A few results for primes in this fractional sequence are proved here.

\begin{proof} (Theorem \ref{thm437.08}) Let $f(n)=\Lambda \left( n^3+2 \right ).$ The condition
\begin{equation}
\sum_{n\leq x}\left | f(n) \right |^2 \ll x^{\alpha}
\end{equation}
is satisfied with $\alpha=1+\varepsilon$ for any small number $\varepsilon>0$. This is Lemma \ref{lem773.01} applied to the irreducible polynomial $g(t)=t^3+2$ of divisor $\tdiv g=1$. Applying Theorem \ref{thm773.01} yield
\begin{equation}
\sum_{n\leq x} \Lambda \left( [x/n]^3+2 \right )
=x\sum_{n\geq 1}\frac{ \Lambda \left( n^3+2 \right ) }{n(n+1) }+O\left ( x^{(2+\varepsilon)/3}\log^2 x \right ).
\end{equation}
The density constant has the lower bound
\begin{eqnarray}
a_3&=&\sum_{n\geq 1}\frac{ \Lambda \left( n^3+2 \right ) }{n(n+1)} \nonumber\\
&\geq &\frac{ \Lambda \left( n_0^3+2 \right ) }{n_0 (n_0+1)} \\
&=&\frac{\log 3}{3} \nonumber,
\end{eqnarray}
where $n_0^3+2=3$ is the first prime in the sequence of primes $ n^3+2$ for $n \geq 1$.  Now, on the contrary, suppose that there are finitely many primes $p=[x/n]^3+2$. Then, there exists a large constant $x_0\geq 1$ such that
\begin{eqnarray} \label{eq437.20}
x_0 \log x_0& \geq &\sum_{n\leq x} \Lambda \left( [x/n]^3+2 \right ) \\
&\geq& 1.002998x+O\left ( x^{(2+\varepsilon)/3}\log^2 x \right ) \nonumber.
\end{eqnarray}
But, since the right side is unbounded as $x \to \infty$, this is a contradiction. Equivalently, it contradicts Theorem \ref{thm773.01}. Therefore, there are infinitely many primes of the form $p=[x/n]^3+2$.
\end{proof}
A small scale numerical experiment gives the estimate
\begin{equation}
a_3\geq \sum_{n\leq 30}\frac{ \Lambda \left( n^3+2 \right ) }{n(n+1)} \geq 1.002998 .
\end{equation}
The fractional prime counting function is defined by
\begin{equation}
\pi_{3}(x)= \#\left \{p= [x/n]^3+2 :p \leq x\text{ and } n \leq x\right  \}.
\end{equation}
\begin{cor} {\normalfont  }  \label{cor437.28}  The fractional prime counting function has the asymptotic formula
\begin{equation}
\pi_{3}(x) \gg\frac{x^{1/3}}{\log x}+O\left ( x^{(2+\varepsilon)/9}\log x \right ).
\end{equation}
\end{cor}

\begin{proof} An application of Theorem \ref{thm437.08}, and partial summation yield
\begin{eqnarray} \label{eq437.40}
\pi_{3}(x)& =&\sum_{\substack{p\leq x\\
p=[x/n]^3+2}}1\\
&=&\sum_{n\leq x^{1/3}}        \frac{ \Lambda \left( [x/n]^3+2 \right )}{\log \left( [x/n]^3+2 \right ) }  \nonumber\\
&\gg& \frac{x^{1/3}}{\log x}+O\left ( x^{(2+\varepsilon)/9}\log x \right ) \nonumber.
\end{eqnarray}
This proves the claim.
\end{proof}


\section{ Primes In Fractional Sequences Of Degree $d$}  \label{s449}
The sequence of integers $\left \{ g(n) :n \geq 1\right  \} \subset \N$ and the associated subsequence of primes has an extensive experimental and conjectural literature, see \cite{BH62}, \cite{FZ18}, \cite{RP96}, \cite{NW00}, et alii. However, there is no literature on the fractional sequence of integers $\left \{ g([x/n]) :n \leq x\right  \} \subset \N$. A few results for primes in this fractional sequence are proved here.

\begin{proof} (Theorem \ref{thm437.75}) Given an irreducible polynomial $g(t)$ of degree $\deg g=O(1)$, and divisor $\tdiv g=1$, let $f(n)=\Lambda \left( |g\left ([x/n]\right) |\right ).$ The condition
\begin{equation}
\sum_{n\leq x}\left | f(n) \right |^2 \ll x^{\alpha}
\end{equation}
is satisfied with $\alpha=1+\varepsilon$ for any small number $\varepsilon>0$, see Lemma \ref{lem773.01}. Applying Theorem \ref{thm773.01} yield
\begin{equation}
\sum_{n\leq x} \Lambda \left( |g\left ([x/n]\right)| \right )
=x\sum_{n\geq 1}\frac{ \Lambda \left( |g(n)| \right ) }{n(n+1) }+O\left ( x^{(2+\varepsilon)/3}\log^2 x \right ).
\end{equation}
Suppose that $p=|g(n_0)|$ is prime for some integer $n_0 \geq 1$. The density constant has the lower bound
\begin{eqnarray}
a_d&=&\sum_{n\geq 1}\frac{ \Lambda \left( |g(n)| \right ) }{n(n+1)} \nonumber\\
&\geq &\frac{ \Lambda \left(|g( n_0)| \right ) }{n_0 (n_0+1)} \\
&=&\frac{\log 2}{2} \nonumber,
\end{eqnarray}
where $g(n_0)\geq 2$ is the first prime in the sequence of primes $ g(n)$ for $n \geq 1$.  Now, on the contrary, suppose that there are finitely many primes $p=g\left ([x/n]\right )$. Then, there exists a large constant $x_0\geq 1$ such that
\begin{eqnarray} \label{eq447.20}
x_0 \log x_0& \geq &\sum_{n\leq x} \Lambda \left( |g\left ([x/n]\right )| \right ) \\
&\geq& a_dx+O\left ( x^{(2+\varepsilon)/3}\log^2 x \right ) \nonumber.
\end{eqnarray}
But, since the right side is unbounded as $x \to \infty$, this is a contradiction. Equivalently, it contradicts Theorem \ref{thm773.01}. Therefore, there are infinitely many primes of the form $p=g\left ([x/n]\right )$.
\end{proof}

The fractional prime counting function is defined by
\begin{equation}
\pi_{d}(x)= \#\left \{p= g\left ([x/n]\right ) :p \leq x\text{ and } n \leq x\right  \}.
\end{equation}
\begin{cor} {\normalfont  }  \label{cor447.28}  The fractional prime counting function has the asymptotic formula
\begin{equation}
\pi_{d}(x) \gg\frac{x^{1/d}}{\log x}+O\left ( x^{(2+\varepsilon)/d}\log x \right ).
\end{equation}
\end{cor}

\begin{proof} An application of Theorem \ref{thm437.75}, and partial summation yield
\begin{eqnarray} \label{eq447.40}
\pi_{d}(x)& =&\sum_{\substack{p\leq x\\
p=g\left ([x/n]\right )}}1\\
&=&\sum_{n\leq x^{1/d}}        \frac{ \Lambda \left( |g\left ([x/n]\right )| \right )}{\log \left( |g\left ([x/n]\right ) |\right ) }  \nonumber\\
&\gg& \frac{x^{1/d}}{\log x}+O\left ( x^{(2+\varepsilon)/d}\log x \right ) \nonumber.
\end{eqnarray}
This proves the claim.
\end{proof}

\section{Twin Primes In Fractional Sequences}  \label{s459}
The sequence of integers pairs $\left \{ n, n+2 :n \leq x\right  \} \subset \N$ and the associated subsequence of twin primes has an extensive literature consisting of partial results, conjectural and numerical data.  Detailed discussions of the twin primes appear in \cite[p.\ 40]{HL23}, \cite[p. 405]{RP96}, \cite[p.395]{FI10}, \cite[p. 342]{NW00}, \cite[p. 25]{PJ09}, et alii. 
\begin{conj} {\normalfont (Twin prime conjecture)} Let $x \geq 1$ be a large number. Let $\Lambda$ be the vonMangoldt function. Then
\begin{equation} \label{eq459.10}
\sum_{n\leq x} \Lambda \left( n \right )\Lambda \left( n+2 \right )
=C_2x+O\left (\frac{ x}{\log x }\right ),
\end{equation}
where the density constant
\begin{equation} \label{eq459.22}
C_2=\prod_{p \geq 3}\left( 1-\frac{1}{(p-1)^2} \right )=0.6601618158468 \ldots .
\end{equation}
\end{conj}
The circle methods heuristics was proposed in \cite[p.\ 46]{HL23}.  However, there is no literature on the fractional sequence of integers pairs $\left \{ [x/n], [x/n]+2 :n \leq x\right  \} \subset \N$. A few results for twin primes in this fractional sequence are proved here.

\begin{thm} \label{thm459.66} Let $x \geq 1$ be a large number and let $\Lambda$ be the vonMangoldt function. 
Then 
\begin{equation}
\sum_{n\leq x} \Lambda \left([x/ n] \right )\Lambda \left([x/ n]+2 \right )
=r_2x+O\left ( x^{(2+\varepsilon)/3}\log^2 x \right ),
\end{equation}
where the density constant is
\begin{equation}
r_2=\sum_{n\geq 1}\frac{ \Lambda \left( n \right ) \Lambda \left( n+2 \right )}{n(n+1) }\geq 0.368142813.
\end{equation}
\end{thm}
\begin{proof} Let $f(n)=\Lambda \left([x/ n] \right )\Lambda \left([x/ n]+2 \right )$. Then, the condition
\begin{equation}
\sum_{n\leq x}\left | f(n) \right |^2  \ll x^{\alpha}
\end{equation}
is satisfied with $\alpha=1+\varepsilon$ for any small number $\varepsilon>0$. This is Lemma \ref{lem773.24} applied to the polynomials $g_1(t)=t$, and $g_2(t)=t+2$. Applying Theorem \ref{thm773.01} yield
\begin{equation} \label{eq459.10}
\sum_{n\leq x} \Lambda \left([x/ n] \right )\Lambda \left([x/ n]+2 \right )
=x\sum_{n\geq 1}\frac{ \Lambda \left( n \right )\Lambda \left( n+2 \right ) }{n(n+1) }+O\left ( x^{(2+\varepsilon)/3}\log^2 x \right ).
\end{equation}
The density constant has the lower bound
\begin{eqnarray} \label{eq459.12}
r_2&=&\sum_{n\geq 1}\frac{\Lambda \left( n \right )\Lambda \left( n+2 \right ) }{n(n+1)} \nonumber\\
&\geq &\frac{ \Lambda \left( n_0 \right )\Lambda \left( n_0+2 \right ) }{n_0 (n_0+1)} \\
&=&\frac{\log 3\log 5}{3\cdot 4} \nonumber,
\end{eqnarray}
where $n_0=3$ and $ n_0+2=5$ is the first odd twin primes in the sequence of primes $ n, n+2$ for $n \geq 1$. Now, on the contrary, suppose that there are finitely many twin primes $p=[x/n]$ and $p+2=[x/n]+2$. Then, there exists a large constant $x_0\geq 1$ such that
\begin{eqnarray} \label{eq437.20}
x_0 \log^2 x_0& \geq &\sum_{n\leq x} \Lambda \left( [x/n] \right )\Lambda \left([x/ n]+2 \right )  \\
&\geq& 0.368142813x+O\left ( x^{(2+\varepsilon)/3}\log^2 x \right ) \nonumber.
\end{eqnarray}
But, since the right side is unbounded as $x \to \infty$, this is a contradiction. Equivalently, it contradicts Theorem \ref{thm773.01}. Therefore, there are infinitely many primes of the form  twin primes $p=[x/n]$ and $p+2=[x/n]+2$.
\end{proof}
A small scale numerical experiment gives the estimate
\begin{equation}
r_2\geq \sum_{n\leq 1000}\frac{ \Lambda \left( n \right )\Lambda \left( n+2 \right ) }{n(n+1)} \geq 0.368142813.
\end{equation}
This estimate does not includes the smallest twin primes $p=2$ and $p=3$. \\

The fractional twin primes counting function is defined by
\begin{equation}
\pi_{T}(x)= \#\left \{p= [x/n], p+2=[x/n]+2 :p \leq x\text{ and } n \leq x\right  \}.
\end{equation}

\begin{cor} {\normalfont  }  \label{cor437.08}  The fractional twin primes counting function has the asymptotic formula
\begin{equation}
\pi_{T}(x) \gg\frac{x}{\log^2 x}+O\left ( x^{(2+\varepsilon)/3} \right ).
\end{equation}
\end{cor}

\begin{proof} An application of Theorem \ref{thm459.66}, and partial summation yield
\begin{eqnarray} \label{eq459.30}
\pi_{T}(x)& =&\sum_{\substack{p\leq x\\
p=[x/n] \text{ and } p=[x/n]+2 \text{ are primes }}}1\\
&=&\sum_{n\leq x}        \frac{ \Lambda \left( [x/n] \right )\Lambda \left([x\ n]+2 \right )}{\log \left( [x/n] \right )\log \left( [x/n]+2 \right ) }  \nonumber\\
&\gg&\frac{x}{\log^2 x}+O\left ( x^{(2+\varepsilon)/3} \right ) \nonumber.
\end{eqnarray}
This proves the claim.
\end{proof}

A more general form of the twin primes conjecture asks for the occurrences of primes pairs $p$ and $p+2m$, with $m\geq 1$ fixed, infinitely often.
\begin{conj} {\normalfont (Polignac conjecture)} Let $x \geq 1$ be a large number. Let $\Lambda$ be the vonMangoldt function. Then
\begin{equation} \label{eq459.10}
\sum_{n\leq x} \Lambda \left( n \right )\Lambda \left( n+2m \right )
=C_mx+O\left (\frac{ x}{\log x }\right ),
\end{equation}
where $m\geq 1$ is a fixed integer, and the density constant
\begin{equation} \label{eq459.22}
C_m=\prod_{p \mid m}\left( \frac{p-1}{p-2 }\right )\prod_{p \geq 3}\left( 1-\frac{1}{(p-1)^2} \right ) .
\end{equation}
\end{conj}

\section{Twin Primes In The Gaussian Ring}  \label{s612}
Let $\alpha=a+ib \in \Z$ be a gaussian integer. The norm $N: \Z[i] \longrightarrow \Z$ function defined by $N(\alpha)=a^2+b^2$ is mapped to an integer. This nice structure facilitates some correspondence between the Guassian twin primes and some rational primes. 

\begin{thm} {\normalfont  }  \label{thm612.12} The Gaussian ring $\Z[i]$ contains an infinitely sequence of twin primes $\pi$ and $\overline{\pi}$ such that $|\pi-\overline{\pi}|=2$. 
\end{thm}
\begin{proof} Let $N(\pi)=n^2+1^2$. By the unique factorization in the gaussian ring it follows that each rational prime $p \equiv 1 \bmod 4$ has a pair of unique gaussian prime factors
\begin{eqnarray}
p&=&n^2+1 \nonumber \\
&=&\left  (n-i\right ) \left ( n+i\right ) \nonumber ,
\end{eqnarray}
up to multiplication by a unit $u \in \mu(4)=\{ -1,1, -i, i\}$. By Theorem \ref{thm437.05} there are infinitely many such prime pairs $\pi=n-i$ and $\overline{\pi}=n+i$. Hence, relation $|\pi-\overline{\pi}|=2$ occurs infinitely often.   
\end{proof}

\section{Germain Primes In Fractional Sequences}  \label{s469}
The sequence of primes pairs $\left \{ p,2p+1 :\text{ prime } p \leq x\right  \} \subset \tP$ is well known as the Sophie Germain primes. It has an extensive literature consisting of partial results, conjectural and numerical data.  Detailed discussions of the Germain primes appear in \cite[p.\ 48]{HL23}, \cite[p. 405]{RP96}, \cite[p. 395]{FI10}, \cite[p. 342]{NW00}, \cite[p. 33]{PJ09}, et alii. 
 \begin{conj} {\normalfont (Germain prime conjecture)} Let $x \geq 1$ be a large number. Let $\Lambda$ be the vonMangoldt function. Then
\begin{equation} \label{eq469.20}
\sum_{n\leq x} \Lambda \left( n \right )\Lambda \left( 2n+1 \right )
=s_1x+O\left (\frac{ x}{\log x }\right ),
\end{equation}
where the density constant
\begin{equation} \label{eq469.22}
s_1=\prod_{p \geq 3}\left( \frac{p(p-1)}{(p-1)^2} \right )=0.6601618158468 \ldots .
\end{equation}
\end{conj}
The circle methods heuristics are introduced in \cite[p.\ 46]{HL23}, \cite{MT06}, et cetera.  However, there is no literature on the fractional sequence of integers pairs $\left \{ [x/p], 2[x/p]+1 :\text{ prime } p \leq x\right  \} \subset \N$. A few results for twin primes in this fractional sequence are proved here.

\begin{thm} \label{thm469.69} Let $x \geq 1$ be a large number and let $\Lambda$ be the vonMangoldt function. 
Then 
\begin{equation}
\sum_{n\leq x} \Lambda \left([x/ n] \right )\Lambda \left(2[x/ n]+1 \right )
=s_1x+O\left ( x^{(2+\varepsilon)/3}\log^2 x \right ),
\end{equation}
where the density constant is
\begin{equation}
s_1=\sum_{n\geq 1}\frac{ \Lambda \left( n \right ) \Lambda \left( 2n+1 \right )}{n(n+1) }\geq 0.620794.
\end{equation}
\end{thm}

\begin{proof}  Let $f(n)=\Lambda \left([x/ n] \right )\Lambda \left(2[x/ n]+1 \right )$. Then, the condition
\begin{equation}
\sum_{n\leq x}\left | f(n) \right |^2  \ll x^{\alpha}
\end{equation}
is satisfied with $\alpha=1+\varepsilon$ for any small number $\varepsilon>0$. This is Lemma \ref{lem773.24} applied to the polynomials $g_1(t)=t$, and $g_2(t)=2t+1$. Applying Theorem \ref{thm773.01} yield
\begin{equation} \label{eq469.10}
\sum_{n\leq x} \Lambda \left([x/ n] \right )\Lambda \left(2[x/ n]+1 \right )
=x\sum_{n\geq 1}\frac{ \Lambda \left( n \right )\Lambda \left( 2n+1 \right ) }{n(n+1) }+O\left ( x^{(2+\varepsilon)/3}\log^2 x \right ).
\end{equation}
The density constant has the lower bound
\begin{eqnarray} \label{eq469.12}
s_1&=&\sum_{n\geq 1}\frac{\Lambda \left( n \right )\Lambda \left( 2n+1 \right ) }{n(n+1)} \nonumber\\
&\geq &\frac{ \Lambda \left( n_0 \right )\Lambda \left( 2n_0+1 \right ) }{n_0 (n_0+1)} \\
&=&\frac{\log 2 \log 5}{2(2+1)} \nonumber,
\end{eqnarray}
where $p_0=2$ and $2p_0+1=5$ is the first pair of Germain primes in the sequence of primes $ p, 2p+1$ for $p \geq 2$. Now, on the contrary, suppose that there are finitely many twin primes $p=[x/n]$ and $q=2[x/n]+1$. Then, there exists a large constant $x_0\geq 1$ such that
\begin{eqnarray} \label{eq469.20}
x_0 \log^2 x_0& \geq &\sum_{n\leq x} \Lambda \left([x/ n] \right )\Lambda \left(2[x/ n]+1 \right )  \\
&\geq& 0.620794x+O\left ( x^{(2+\varepsilon)/3}\log^2 x \right ) \nonumber.
\end{eqnarray}
But, since the right side is unbounded as $x \to \infty$, this is a contradiction. Equivalently, it contradicts Theorem \ref{thm773.01}. Therefore, there are infinitely many primes of the form  Germain primes $p=[x/n]$ and $q=2[x/n]+1$.
\end{proof}
A small scale numerical experiment gives the estimate
\begin{equation}
s_1\geq \sum_{n\leq 1000}\frac{ \Lambda \left( n \right )\Lambda \left( 2n+1 \right ) }{n(n+1)} \geq 0.620794742886735.
\end{equation}
This estimate does not includes the smallest twin primes $p=2$ and $p=3$. \\

The fractional Germain primes counting function is defined by
\begin{equation}
\pi_{S}(x)= \#\left \{p= [x/n], q=2[x/n]+1 :p \leq x\text{ and } n \leq x\right  \}.
\end{equation}

\begin{cor} {\normalfont  }  \label{cor469.47}  The fractional twin primes counting function has the asymptotic formula
\begin{equation}
\pi_{S}(x) \gg\frac{x}{\log^2 x}+O\left ( x^{(2+\varepsilon)/3} \right ).
\end{equation}
\end{cor}

\begin{proof} An application of Theorem \ref{thm469.69}, and partial summation yield
\begin{eqnarray} \label{eq469.30}
\pi_{S}(x)& =&\sum_{\substack{p\leq x\\
p=[x/n] \text{ and } p=2[x/n]+1 \text{ are primes }}}1\\
&=&\sum_{n\leq x}        \frac{ \Lambda \left( [x/n] \right )\Lambda \left( 2[x/n]+1 \right )}{\log \left( [x/n] \right )\log \left( 2[x/n]+1 \right ) }  \nonumber\\
&\gg&\frac{x}{\log^2 x}+O\left ( x^{(2+\varepsilon)/3} \right ) \nonumber.
\end{eqnarray}
This proves the claim.
\end{proof}

\section{Distribution of the Fractional Parts}  \label{s238}
The statistical properties of the fractional parts of various sequences of real numbers are of interest in the mathematical sciences. In the case of the sequence of primes $\tP=\{2,3,5,7, \ldots \}$, one of the earliest result, due to DelaValle Poussin, dealing with the fractional parts of primes provides a completely determined asymptotic part:
\begin{equation}
\sum_{n \leq x}\left \{\frac{x}{p}\right  \}=(1-\gamma)\frac{x}{\log x}+O\left ( \frac{x}{\log^{2}x} \right ),
\end{equation}
where $\gamma=.5772 \ldots $ is a constant. Many other results are also available in the literature.\\

For the individual primes $p \in \tP$, one of the best known result for the fractional prat of the square root claims that 
\begin{equation}\label{eq238.12}
\left \{\sqrt{p}\right  \}<\frac{c}{p^{1/4+\varepsilon}},
\end{equation}
with $c>0$ constant, and $\varepsilon >0$ arbitrarily small. This is proved in \cite{BA83}, and \cite{HG83}. This implies that the sum of the fractional parts is
\begin{equation} \label{eq238.18} 
\sum_{p \leq x }\left \{p^{1/2}\right  \}=O\left (x^{3/4-\varepsilon}  \right)
\end{equation}
for all sufficiently large $x \geq 1$. \\

In this section some works are considered for some sequences of primes in fractional sequences. The proofs are based on the relevant theorems, and elementary concepts as the binomial series expansion
\begin{equation} \label{eq238.99}
\left (1+x \right )^{1/m}=1+\frac{1}{mx}+O\left (\frac{1}{x^{2} }\right ) ,
\end{equation}
for $|x|<1$. In particular, these elementary methods actually improve (\ref{eq238.12}) to the sharper upper bound
\begin{equation}\label{eq238.22}
\left \{\sqrt{p}\right  \}<\frac{c}{p^{1/2}},
\end{equation}
unconditionally, for example, set $\alpha=2$ in Lemma \ref{lem238.21}.

\subsection{Fractional Parts In A Sequence Of Piatetski-Shapiro Primes}
The distribution of the sequence of primes $\left \{p=\left [ n^{\beta} \right ]+1: n \geq 1\right  \}\subset \tP,$
where $\beta \in [1, 12/11]$ is a real number, has a large literature. However, there is no literature on the fractional parts $\left \{p^{1/\alpha}\right  \}$ of these primes.\\

\begin{lem} \label{lem238.21} Let $\alpha >1$ be a real number. Given $\beta \in [1, 12/11]$, there are infinitely many primes $p=
\left  [ n^{\beta} \right ]+1$, $ n \geq 1$, such that the fractional parts satisfy the inequality
\begin{equation}
\left \{p^{1/\alpha}\right  \}=O\left (\frac{1}{p^{1-1/\alpha}} \right ).
\end{equation}
\end{lem}
\begin{proof} Take the prime number $p=\left  [ n^{\beta} \right ]+1$, with $n \geq 1$. By definition, and the binomial series expansion (\ref{eq238.99}), this is precisely
\begin{eqnarray}
\left \{p^{1/\alpha}\right  \}&=&p^{1/\alpha}-\left [ p^{1/\alpha} \right ] \nonumber \\
&=&\left (\left  [ n^{\beta} \right ]+1\right )^{1/\alpha}-\left [ \left (\left  [ n^{\beta} \right ]+1\right )^{1/\alpha} \right ] \nonumber \\
&\leq&\left  [ n^{\beta} \right ]^{1/\alpha}      \left (1+\frac{1}{\left  [ n^{\beta} \right ]}\right )^{1/\alpha}-\left  [ n^{\beta} \right ]^{1/\alpha}   \\
&\leq &\left  [ n^{\beta} \right ]^{1/\alpha}      \left (1+\frac{1}{ \alpha \left [  n^{\beta} \right ]}  +O\left (\frac{1}{n^{2\beta}} \right ) \right )-\left  [ n^{\beta} \right ]^{1/\alpha} \nonumber \\
&\ll& \frac{1}{n^{\beta(1-1/\alpha)}}+O\left (\frac{1}{n^{\beta(2-1/\alpha)}} \right )  \nonumber \\
&=&O\left (\frac{1}{p^{1-1/\alpha}} \right )  \nonumber.
\end{eqnarray}
The last inequality follows from $n^{\beta} \leq p \leq n^{\beta}+1$. By the Piatetski-Shapiro prime number theorem, confer (\ref{eq437.16}), it follows that this inequality occurs infinitely often.   
\end{proof}

\begin{cor} \label{cor238.25} Let $\alpha >1$ be a real number. Given the sequence of primes  $\left \{ p=\left  [ n^{\beta} \right ]+1: n \geq 1 \right  \}$, the sum of the fractional parts is
\begin{equation}
\sum_{p \leq x }\left \{p^{1/\alpha}\right  \}=O\left (\frac{1}{x^{1-1/\alpha-1/\beta} \log x} \right)
\end{equation}
for all sufficiently large $x \geq 1$.
\end{cor}

\begin{proof} Let \(\pi_{\beta}(x)=\#\{ p=\left  [ n^{\beta} \right ]+1 \leq x\}\leq 2x^{1/\beta} \log^{-1}x\), confer (\ref{eq437.16}). Now, consider the finite sum
\begin{eqnarray}
\sum_{n\leq x} \frac{1}{p^{1-1/\alpha}}
&\leq& \int_{1}^{x}\frac{1}{t^{1-1/\alpha}} d \pi_{\beta}(t) \nonumber\\
&\leq& \frac{\pi_{\beta}(x)}{x^{1-1/\alpha}}+c\int_{1}^{x}\frac{\pi_{\beta}(t)}{t^{2-1/\alpha}} d t\\
&=&O\left (\frac{1}{x^{1-1/\alpha-1/\beta} \log x} \right)\nonumber,
\end{eqnarray}
where $c>0$ is a constant. The claim follows 
\begin{equation}
\sum_{p \leq x} \left \{p^{1/\alpha}\right  \} \ll \sum_{p \leq x} \left (\frac{1}{p^{1-1/\alpha}} \right ).
\end{equation}
and Lemma \ref{lem238.21}.  
\end{proof}

\subsection{Fractional Parts In Quadratic Sequences}
The distribution of the sequence of primes $\left \{ p = n^2 + 1: n \geq 1 \right  \}$, and the fractional parts $\left \{p^{1/2}\right  \}$ of these primes are equivalent problems. 

\begin{lem} \label{lem238.41} There are infinitely many primes $p\geq 2$ such that the fractional parts satisfy the inequality
\begin{equation}
\left \{\sqrt{p}\right  \}<\frac{c_2}{\sqrt{p}},
\end{equation}
with $c_2>1/2$ constant.
\end{lem}
\begin{proof} Take the prime number $p = n^2 + 1$, with $n \geq 1$. By definition, and the binomial series expansion (\ref{eq238.99}), this is precisely
\begin{eqnarray}
\left \{\sqrt{p}\right  \}&=&\sqrt{p}-\left [ \sqrt{p} \right ] \nonumber \\
&=&\sqrt{n^2+1}-\left [ \sqrt{n^2+1} \right ] \nonumber \\
&=&\sqrt{n^2(1+1/n^2)} -n  \\
&=&n \left ( 1+\frac{1}{2n^2}+O\left (\frac{1}{n^4} \right ) \right )-n \nonumber \\
&=& \frac{1}{2n}+O\left (\frac{1}{n^3} \right )  \nonumber \\
&<&\frac{c_2}{\sqrt{p}} \nonumber,
\end{eqnarray}
where $c_2 > 1/2$ is a constant. By Corollary \ref{cor437.08}, it follows that this inequality occurs infinitely often.   
\end{proof}

\begin{cor} \label{cor238.45} Given the sequence of primes  $\left \{ p = n^2 + 1, n \geq 1 \right  \}$, the sum of the fractional parts is
\begin{equation}
\sum_{\substack{p \leq x,\\ p=n^2+1} }\left \{\sqrt{p}\right  \}=O\left (\log \log x \right)
\end{equation}
for all sufficiently large $x \geq 1$.
\end{cor}

\begin{proof} Let \(\pi_2(x)=\#\{ p = n^2 + 1 \leq x\}\leq 2x^{1/2} \log^{-1}x\), see Theorem \ref{thm437.05}. Now, consider the finite sum
\begin{eqnarray}\sum_{\substack{p \leq x,\\ p=n^2+1} } \frac{1}{\sqrt{p}}
&=& \int_{1}^{x}\frac{1}{t^{1/2}} d \pi_2(t) \nonumber\\
&=& \frac{\pi_2(x)}{x^{1/2}}+\int_{1}^{x}\frac{\pi_2(t)}{t^{3/2}} d t\\
&=&O\left (\log \log x \right)\nonumber.
\end{eqnarray}
The claim follows from 
\begin{equation}
\sum_{p \leq x} \left \{\sqrt{p}\right  \}<\sum_{p \leq x}\frac{c_2}{\sqrt{p}},
\end{equation}
and Lemma \ref{lem238.41}. 

\end{proof}

\subsection{Fractional Parts In Cubic Sequences}
The distribution of the sequence of primes $\left \{ p = n^3+2: n \geq 1 \right  \}$, and the fractional parts $\left \{p^{1/3}\right  \}$ of these primes are equivalent problems. There is no known result on this case. 

\begin{lem} \label{lem238.81} There are infinitely many primes $p\geq 2$ such that the fractional parts satisfy the inequality
\begin{equation}
\left \{p^{1/3} \right  \}<\frac{c_3}{p^{2/3}},
\end{equation}
with $c_3>1/2$ constant.
\end{lem}
\begin{proof} Take the prime number $p = n^3+2$, with $n \geq 1$. By definition, and the binomial series expansion (\ref{eq238.99}), this is precisely
\begin{eqnarray}
\left \{p^{1/3} \right  \}&=&p^{1/3}-\left [ p^{1/3} \right ] \nonumber \\
&=&\left (n^3+2 \right )^{1/3}-\left [ \left (n^3+2 \right )^{1/3} \right ] \nonumber \\
&=&\left (n^3+2 \right )^{1/3} -n  \\
&=&n \left ( 1+\frac{2}{3n^3}+O\left (\frac{1}{n^6} \right ) \right )-n \nonumber \\
&=& \frac{2}{3n^2}+O\left (\frac{1}{n^5} \right )  \nonumber \\
&<&\frac{c_3}{p^{2/3}}\nonumber,
\end{eqnarray}
where $c > 2/3$ is a constant. By Corollary \ref{cor437.28}, it follows that this inequality occurs infinitely often.   
\end{proof}

\begin{cor} \label{cor238.85} Given the sequence of primes  $\left \{ p = n^3+2, n \geq 1 \right  \}$, the sum of the fractional parts is
\begin{equation}
\sum_{\substack{p \leq x,\\ p=n^3+2} }\left \{p^{1/3} \right  \}=O\left (\frac{1}{x^{1/3} \log x} \right)
\end{equation}
for all sufficiently large $x \geq 1$.
\end{cor}

\begin{proof} Let \(\pi_3(x)=\#\{ p = n^3+2 \leq x\}\leq 2x^{1/3} \log^{-1}x\), see Theorem \ref{thm437.05}. Now, consider the finite sum
\begin{eqnarray}
\sum_{\substack{p \leq x,\\ p=n^3+2} } \frac{1}{p^{2/3}}
&=& \int_{1}^{x}\frac{1}{t^{2/3}} d \pi_3(t) \nonumber\\
&=& \frac{\pi_3(x)}{x^{2/3}}+\int_{1}^{x}\frac{\pi_3(t)}{t^{5/3}} d t\\
&=&O\left (\frac{1}{x^{1/3} \log x} \right)\nonumber.
\end{eqnarray}
The claim follows from 
The claim follows from 
\begin{equation}
\sum_{p \leq x} \left \{p^{1/3}\right  \}<\sum_{p \leq x}\frac{c_3}{p^{2/3}},
\end{equation}
and Lemma
Lemma \ref{lem238.81}.  
\end{proof} 

\section{Exercises}
\begin{exe} { \normalfont Let $\alpha=\sqrt{2} $. Compute the first 100 primes in the set
$$\mathcal{B}_{\alpha}=\left \{ p = [\alpha n ]: n \geq 1\right  \}. $$ 
}
\end{exe}

\begin{exe} { \normalfont Let $\alpha=\pi $. Compute the first 100 primes in the set
$$\mathcal{B}_{\alpha}=\left \{ p = [\alpha n ]: n \geq 1\right  \}. $$ 
}
\end{exe}

\begin{exe} { \normalfont Explain some of the differences between the subset of primes 
\begin{enumerate}[label=(\alph*)]
\item $\mathcal{B}_{\alpha}=\left \{ p = [\alpha n ]: n \geq 1\right  \} $, where $\alpha >0$ is an algebraic irrational, and the subset of primes
\item $\mathcal{B}_{\beta}=\left \{ p = [\beta n ]: n \geq 1\right  \} $, where $\beta >0$ is a nonalgebraic irrational.
\end{enumerate}
}
\end{exe}

\begin{exe} { \normalfont Explain some of the differences between the subset of primes 
\begin{enumerate}[label=(\alph*)]
\item $\mathcal{A}_{\alpha}=\left \{ p = [n^{\alpha} ]: n \geq 1\right  \} $, where $\alpha >0$ is a rational, 
the subset of primes 
\item $\mathcal{A}_{\beta}=\left \{ p = [n^{\beta} ]: n \geq 1\right  \} $, where $\beta >0$ is an algebraic irrational, and  
 the subset of primes 
\item
$\mathcal{A}_{\gamma}=\left \{ p = [n^{\gamma} ]: n \geq 1\right  \} $, where $\gamma >0$ is a nonalgebraic irrational.\end{enumerate}
}
\end{exe}
\begin{exe} { \normalfont Determine the least prime in the following fractional sequences, (and compare these results to the result in \cite{SJ08}): 
\begin{enumerate}[label=(\alph*)]
\item $\mathcal{A}_{\alpha}=\left \{ p = [n^{\alpha} ]: n \geq 1\right  \} $, where $\alpha >0$ is a rational, 
the subset of primes 
\item $\mathcal{A}_{\beta}=\left \{ p = [n^{\beta} ]: n \geq 1\right  \} $, where $\beta >0$ is an algebraic irrational, and  
 the subset of primes 
\item
$\mathcal{A}_{\gamma}=\left \{ p = [n^{\gamma} ]: n \geq 1\right  \} $, where $\gamma >0$ is a nonalgebraic irrational.\end{enumerate}
}
\end{exe}

\begin{exe} { \normalfont Let $x \geq 1$ be a large number. Explain some of the differences between the subset of primes 
\begin{enumerate}[label=(\alph*)]
\item $\mathcal{Q}_{0}=\left \{ p = [x/n]^{2} +1: n \leq x\right  \} $, and the subset of primes
\item $\mathcal{Q}_{\alpha}=\left \{ p = [\alpha n^2 ]+1: n \leq x\right  \} $, where $\alpha >0$ is irrational.\end{enumerate}
}
\end{exe}

\begin{exe} {\normalfont Let $f $ be an arithmetic function, let $ x \geq 1$ be a large number, and let $[x]=x-\{x\}$ denotes the largest integer function. Extend the Lagrange fractional formula to finite sums over the primes:
$$\sum_{p\leq x}f \left( [x/p] \right )  =\sum_{p\leq x}f(p) \delta(p),$$
where $p \leq x$ ranges over the prime numbers, and the fudge factor $\delta(p)$ is
$$\delta (p) \stackrel{?}{=}\left ( \left [ \frac{x}{p} \right ]-\left [\frac{x}{p+1} \right ] \right )\left ( \left [ \frac{x}{p+a} \right ]-\left [\frac{x}{p+b)} \right ] \right ),$$ $a\geq 0$ and $b \geq 0$ are fixed integers.}
\end{exe}

\begin{exe} \label{exe933.90}{\normalfont Construct a polynomial $f(x) \in \Z[x]$ that generates 10 primes. Hint: In general, interpolation, as Lagrange interpolation, solve this problem. For example, the polynomial $f(x)=(g(x)+1)(x^2+1)$, where $g(x)=(x-1)(x-2)(x-4)(x-6)(x-10)$, generated 5 primes $f(n)=2, 5, 17, 37, 101$ for $1 \leq n \leq 10$.}
\end{exe}

\end{document}